\documentclass{article}
\usepackage[margin=30truemm]{geometry}
\usepackage{graphicx} 
\usepackage{cite}
\usepackage[pdfencoding=auto, psdextra]{hyperref}
\usepackage{amsmath,amsthm, amssymb,amsfonts}
\usepackage{algorithmic}
\usepackage{textcomp}
\usepackage{xcolor}
\usepackage{appendix}

\newtheorem{theorem}{Theorem}

\newtheorem{lemma}[theorem]{Lemma}

\newtheorem{problem}[theorem]{Problem}
\newtheorem{proposition}[theorem]{Proposition}
\newtheorem{corollary}[theorem]{Corollary}
\newtheorem{conjecture}[theorem]{Conjecture}
\newtheorem{definition}[theorem]{Definition}
\newtheorem{remark}[theorem]{Remark}
\numberwithin{equation}{section}

\newcommand{\F}{\mathbb{F}}
\newcommand{\Ext}{{\rm Ext}}
\newcommand{\poly}{{\rm poly}}

\title{On the Paley RIP and Paley graph extractor}
\author{Shohei Satake
\thanks{
Research and Education Institute for Semiconductors and Informatics, Kumamoto University\\
2-39-1, Kurokami, Chuo, Kumamoto, 860-8555, Japan.\\
E-mail: shohei-satake@kumamoto-u.ac.jp
}
}
\date{}

\begin{document}

\maketitle
\begin{abstract}
Constructing explicit RIP matrices is an open problem in compressed sensing theory. In particular, it is quite challenging to construct explicit RIP matrices that break the square-root bottleneck. On the other hand, providing explicit $2$-source extractors is a fundamental problem in theoretical computer science, cryptography and combinatorics. Nowadays, there are only a few known constructions for explicit $2$-source extractors (with negligible errors) that break the half barrier for min-entropy.

In this paper, we establish a new connection between RIP matrices breaking the square-root bottleneck and $2$-source extractors breaking the half barrier for min-entropy. 
Here we focus on an RIP matrix (called the Paley ETF) and a $2$-source extractor (called the Paley graph extractor), where both are defined from quadratic residues over the finite field of odd prime order $p\equiv 1 \pmod{4}$. As a main result, we prove that if the Paley ETF breaks the square-root bottleneck, then the Paley graph extractor breaks the half barrier for min-entropy as well. Since it is widely believed that the Paley ETF breaks the square-root bottleneck, our result accordingly provides a new affirmative intuition on the conjecture for the Paley graph extractor by Benny Chor and Oded Goldreich.
\end{abstract}

\section{Introduction}
In this paper, we establish a new relation between two major conjectures (Conjectures~\ref{conj-PaleyRIP} and \ref{conj-PaleyExt} below) in compressed sensing and theoretical computer science, which both are regarding structures defined from {\it quadratic residues} over finite fields.

\subsection{RIP matrices}
Matrices with {\it restricted isometry property} ({\it RIP}) defined below have fundamental applications to signal processing since, by adopting them, it is possible to measure and recover sparse signals using significantly fewer measurements than the dimension of the signals~\cite{Ca2008}. 

\begin{definition}[Restricted isometry property, RIP]
\label{def-RIP}
Let $\Phi$ be a complex $M \times N$ matrix. Suppose that $K \leq M \leq N$ and $0 \leq \delta<1$. 
Then $\Phi$ is said to have the {\it $(K, \delta)$-restricted isometry property} ({\it RIP}) if 
\begin{equation*}
    (1-\delta)||\mathbf{x}||^2 \leq ||\Phi \mathbf{x}||^2 \leq  (1+\delta)||\mathbf{x}||^2
\end{equation*}
for every $N$-dimensional complex vector $\mathbf{x}$ with at most $K$ non-zero entries. Here $||\cdot||$ denotes the $\ell_2$ norm.
\end{definition}

According to Defintion~\ref{def-RIP}, the following RIP constant is an important constant showing the trade-off of $K$ and $\delta$.

\begin{definition}[RIP constant]
\label{def-RIPconst}
For each $K\geq 1$ and an $M\times N$ matrix, the {\it RIP constant} $\delta_K=\delta_K(\Phi)$ is the minimum real number $\delta \geq 0$ that $\Phi$ has the $(K, \delta)$-RIP.
\end{definition}

On the other hand, it is known that the problem checking whether a given matrix has RIP is NP-hard~\cite{BDMS2013}.
Thus many publications have attempted to give {\it explicit} constructions of matrices having RIP. 
However, for most of known such constructions, it is only proved that $(K, \delta)$-RIP with $K=O(\sqrt{M})$, following from the RIP evaluation via the mutual coherence and the Welch bound~\cite{W1974} (see also \cite{BDFKK2011}). 
This barrier for the magnitude of the order of $K$ is called the {\it square-root bottleneck} or {\it quadratic bottleneck}.
From this situation, the following problem arises.

\begin{problem}[\cite{BDFKK2011}]
\label{prob-beatsqrt}
Construct an explicit $M \times N$ matrix $\Phi$ having the $(K, \delta)$-RIP with $K=\Omega(M^{\gamma})$ for some $\gamma>1/2$. 
\end{problem}
To our best knowledge, the first (unconditional) solution to this problem was given by Bourgain, Dilworth, Ford, Konyagin and Kutzarova~\cite{BDFKK2011}.
By applying recent progresses on additive combinatorics, the authors of \cite{FKS2022} provided the state-of-the-art construction for explicit $M\times N$ matrices having $(M^{1/2+\kappa}, M^{-\mu})$-RIP with $\kappa=0.815\times 10^{-7}$ and some $\mu>0$.

On the other hand, there is another good candidate for RIP matrix, namely, the {\it Paley ETF} $\Phi_p$ which is a $(p+1)/2 \times (p+1)$ matrix defined by quadratic residues modulo a prime $p\equiv 1 \pmod{4}$ (see Definition~\ref{def-paley-matrix}).
Indeed, it was conjectured in \cite{BFMW2013} (see also \cite{BMM2017}, \cite{S2021}) that the RIP of $\Phi_p$ would be superior than the state-of-the-art construction.

\begin{conjecture}[\cite{BFMW2013}, see also Conjecture~\ref{conj-RIPconst} in Section~\ref{sect-RIPconst}]
\label{conj-PaleyRIP}
For any sufficiently large prime $p\equiv 1 \pmod{4}$,
the Paley ETF $\Phi_p$ has the $(K, \delta)$-RIP with
$$
K=\frac{c_1 p}{\log^{c_2} p}
$$
for constant numbers $c_1, c_2>0$ and $0<\delta<\sqrt{2}-1$.
\end{conjecture}


\subsection{Randomness extractors}

A {\it randomness extractor} is a fundamental object in theoretical computer science and cryptography.
Informally, a randomness extractor is a deterministic function which takes inputs from an unknown source (with sufficient min-entropy, defined below) and outputs bits over nearly-uniform distribution.
Hereafter let $\Omega$ be a non-empty finite set and $n=\lceil \log_2 |\Omega| \rceil$ (that is, $n$ is the least number of bits to express the size of $\Omega$).

\begin{definition}[Min-entropy]
For a source $X$ on $\Omega$, its {\it min-entropy} $H_{\infty}(X)$ is defined as
\begin{align*}
    H_{\infty}(X):=\min_{x\in \Omega}(-\log \Pr[X=x]).
\end{align*}
\end{definition}
If $H_{\infty}(X)$ is small, then the corresponding distribution is ``far'' from the uniform distribution, which is typical in the scenario of computer science and cryptography. 
This is why we want to deal with sources having {\it small} min-entropy. 

One of well-studied randomness extractors is a {\it $2$-source extractor} which takes inputs from two independent sources.
Towards the precise definition of $2$-source extractors, we also introduce the statistical distance which measures how two probability distributions are close.

\begin{definition}[Statistical distance]
Let $\mathcal{D}_1$ and $\mathcal{D}_2$ be two probability distributions over $\Omega$.
Then the {\it statistical distance} between $\mathcal{D}_1$ and $\mathcal{D}_2$ is defined as
\begin{align*}
  |\mathcal{D}_1-\mathcal{D}_2|
  :=\frac{1}{2}\sum_{d \in \Omega}
  \Bigl| \Pr[\mathcal{D}_1=d]-\Pr[\mathcal{D}_2=d] \Bigr|.  
\end{align*}
For $\epsilon>0$, we write $\mathcal{D}_1 \approx_{\epsilon} \mathcal{D}_2$ if $|\mathcal{D}_1-\mathcal{D}_2|\leq \epsilon$.
\end{definition}

Now we introduce the precise definition of $2$-source extractors.
In this paper we focus on $2$-source extractors with $1$-bit outputs.
\begin{definition}[$2$-source extractors with $1$-bit outputs]
A function $\Ext: \Omega \times \Omega \to \{0, 1\}$ is a {\it $2$-source extractor for min-entropy $k$ and error $\epsilon$} if for any two independent sources $X$ and $Y$ on $\Omega$ with $H_{\infty}(X)\geq k$ and $H_{\infty}(Y)\geq k$, we have
\begin{align*}
    \Ext(X,Y) \approx_\epsilon \mathcal{U}
\end{align*}
where $\mathcal{U}$ denotes the uniform distribution over $\{0,1\}$.
Here we say $\Ext$ is a {\it $2$-source $(k,\epsilon)$-extractor}.
\end{definition}

Note that for applications to computer science and cryptography, it is desired that the error $\epsilon$ is {\it negligible}, that is $\epsilon <\exp(-cn)$ for some $c>0$.
Whereas it is well-known that a probabilistic argument shows the existence of $2$-source $(k,\epsilon)$-extractors for $k\geq \log n+2\log(1/\epsilon)+1$, it is open and quite challenging to provide {\it explicit} $2$-source $(k,\epsilon)$-extractors for small $k$, and, hopefully, negligible $\epsilon$.
In \cite{CG1988}, Chor and Goldreich constructed explicit $2$-source $(k,\epsilon)$-extractors for $k>n/2$ and negligible $\epsilon$, and since then there was no constructions for $k\leq n/2$ for two decades.
This barrier for $k$ is referred as the {\it half barrier}, and the following problem arises.  

\begin{problem}[e.g. \cite{B2005}, \cite{CZ2019}, \cite{CG1988}]
Provide explicit $2$-source $(k,\epsilon)$-extractors for $k<n/2$ and negligible $\epsilon$.
\end{problem}

The first affirmative solution was provided by Bourgain~\cite{B2005} based on additive combinatorics, where he constructed explicit $2$-source $((1/2-\sigma)n,\epsilon)$-extractors for some small constant $\sigma>0$ and negligible $\epsilon$ (see also \cite{Rao2007}). In \cite{L2019} Lewko presented explicit $2$-source $(4n/9,\epsilon)$-extractors for negligible $\epsilon$ (see also \cite{L2020}).
On the other hand, if one allows {\it non-negligible} error $\epsilon$ (i.e. $\epsilon<n^{-c}$ for some $c>0$, even or $\epsilon<1$ is a constant), then there are constructions for much smaller min-entropy, which provide the best known construction for Ramsey graphs (e.g. \cite{CZ2019}).
A significant breakthrough was obtained by Chattopadhyay and Zuckerman~\cite{CZ2019} who constructed explicit $2$-source $(k,\epsilon)$-extractors for $k=\log^{56} (n/\epsilon)$, where $\epsilon$ is non-negligible.
Recently Li~\cite{L2023} constructed explicit $2$-source $(k,\epsilon)$-extractors for $k>c\log n$ for some $c>1$ and any constant $\epsilon>0$.

Now let $\Omega$ be the field $\F_p$ of odd prime order $p$ (and hence $n=\lceil \log_2 p \rceil$). 
Then the {\it Paley graph extractor} $\Ext_p$ is a $2$-source extractor defined by the quadratic multiplicative character of $\F_p$ (see Definition~\ref{def-PaleyExt}).
This extractor was introduced by Chor and Goldreich~\cite{CG1988} who made the following conjecture.  

\begin{conjecture}[\cite{CG1988}]
\label{conj-PaleyExt}
The Paley graph extractor $\Ext_p$ is a $2$-source $(\alpha n, \epsilon)$-extractor for {\it any} $0<\alpha<1$ and some negligible $\epsilon$
\end{conjecture}

Although this conjecture is affirmatively confirmed only for $\alpha>1/2$ (\cite{CG1988}), it is even believed (e.g. \cite{M2016}) in number theory that $\Ext_p$ is a $(\poly \log n, \epsilon)$-extractor where $\epsilon=o(1)$.
Hence, even after the state-of-the-art constructions (e.g. \cite{CZ2019}, \cite{L2023}), investigating the Paley graph extractor is still very interesting and quite challenging.

\subsection{Main results}
In this paper, we focus on the Paley ETF and Paley graph extractor, and establish a new relationship between Conjectures~\ref{conj-PaleyRIP} and \ref{conj-PaleyExt} which are seemingly unrelated and have been investigated independently.
The main results of this paper are summarized as follows.

\begin{itemize}
    \item {\it If the Paley ETF $\Phi_p$ is an RIP matrix breaking the square-root bottleneck, then the Paley graph extractor $\Ext_p$ is a $2$-source extractor breaking the half barrier, having a negligible error (see Theorem~\ref{thm-main-intro}, Corollary~\ref{cor-2}). \\
    \item Furthermore, Conjecture~\ref{conj-PaleyRIP} (more precisely, Conjecture~\ref{conj-RIPconst} in Section~\ref{sect-RIPconst}) implies Conjecture~\ref{conj-PaleyExt} (see Corollary~\ref{cor-4}).}
\end{itemize}

The rest of this paper is organized as follows.
In Section~\ref{sect-prelim}, we briefly review some necessary mathematical notations as well as the precise definitions for the Paley ETF and Paley graph extractor.
Sections~\ref{sect-main}, \ref{sect-main-tech} and \ref{sect-RIPconst} consists of the main body of this paper.
Section~\ref{sect-main} presents the first main result (Theorem~\ref{thm-main-intro}) and its corollaries, and Section~\ref{sect-main-tech} is devoted for the proof of Theorem~\ref{thm-main-intro}.
Finally, Section~\ref{sect-RIPconst} presents the second main result (Corollary~\ref{cor-4}).

\section{Preliminaries}
\label{sect-prelim}
Let $p$ denote an odd prime number. 
Let $\mathbb{F}_p$ be the finite field with $p$ elements which can be identified to the residue ring $\mathbb{Z}/p\mathbb{Z}$.
It is well-known that the multiplicative group of $\mathbb{F}_p$, denoted by $\mathbb{F}_p^*$, is a cyclic group of order $p-1$, consisting of all non-zero elements of $\mathbb{F}_p$.
A non-zero element $a \in \mathbb{F}_p$ is called a {\it quadratic residue modulo $p$} if the equation $X^2 \equiv a \pmod{p}$ has non-zero solutions.
Note that there exist exactly $(p-1)/2$ quadratic residues modulo $p$.
The {\it canonical additive character} $\psi$ of $\mathbb{F}_p$ is a map from $\mathbb{F}_p$ to the unit circle in $\mathbb{C}$ such that $\psi(x):=\exp(\frac{2 \pi \sqrt{-1}}{p} \cdot x)$ for all $x \in \mathbb{F}_p$.
Notice that for every pair of $x, y \in \mathbb{F}_p$, we have $\psi(x+y)=\psi(x)\psi(y)$.
A {\it quadratic multiplicative character} $\chi$ of $\mathbb{F}_p$ is a map from $\mathbb{F}_p$ to $\{0, \pm 1\}$ defined as
\vspace{-0.1cm}
\begin{equation*}
\chi(x):=
    \begin{cases}
    0  & x=0;\\
    1  & \text{$x$ is a quadratic residue modulo $p$};\\
    -1 & \text{otherwise}.
    \end{cases}
\end{equation*}
Notice that $\chi(xy)=\chi(x)\chi(y)$ for every pair of $x, y \in \mathbb{F}_p$.

Also we introduce the following evaluation for the quadratic Gauss sum.
\begin{lemma}[\cite{LN1994}]
\label{lem-Gauss}
For a prime $p\equiv 1 \pmod{4}$ and $a \in \mathbb{F}_p^*$, 
\begin{equation*}
\sum_{x \in \mathbb{F}_p}\psi(ax^2)=\chi(a)\sqrt{p}.
\end{equation*}
\end{lemma}

Now we are ready to define the Paley ETF.

\begin{definition}[Paley ETF, \cite{BMM2017}, \cite{R2007}, \cite{Z1999}]
\label{def-paley-matrix}
For a prime $p\equiv 1 \pmod{4}$ let $Q_p$ be the set of quadratic residues modulo $p$. 
Suppose that elements of $\mathbb{F}_p$ and $Q_p$ are labelled as $\mathbb{F}_p=\{0=a_1, a_2, \ldots, a_p\}$ and $Q_p=\{b_1, b_2, \ldots, b_{(p-1)/2}\}$, respectively.
Then the {\it Paley ETF} $\Phi_p$ is a $(p+1)/2 \times (p+1)$ complex matrix of the following form.
\begin{equation*}
  \Phi_p := 
  \left[
    \begin{array}{ccccc}
      \frac{1}{\sqrt{p}} & \frac{1}{\sqrt{p}} & \ldots & \frac{1}{\sqrt{p}} & 1 \\ \\
      \sqrt{\frac{2}{p}} & \sqrt{\frac{2}{p}}\psi(b_1a_{2}) & \ldots & \sqrt{\frac{2}{p}}\psi(b_1a_{p}) & 0 \\ \\
      \sqrt{\frac{2}{p}} & \sqrt{\frac{2}{p}}\psi(b_2a_{2}) & \ldots & \sqrt{\frac{2}{p}}\psi(b_2a_{p}) & 0 \\ \\
      \vdots & \vdots & \ddots & \vdots & \vdots \\ \\
      \sqrt{\frac{2}{p}} & \sqrt{\frac{2}{p}}\psi \bigl(b_{\frac{p-1}{2}}a_{2} \bigr) & \ldots & \sqrt{\frac{2}{p}}\psi \bigl(b_{\frac{p-1}{2}}a_{p} \bigr) & 0
    \end{array}
  \right]
\end{equation*}
\end{definition}
As shown in \cite{R2007}, $\Phi_p$ is an equiangular tight frame (ETF). 

\begin{lemma}
\label{lem-inner}
Let $\phi_i$ be the $i$-th column of $\Phi_p$. Then, for each $1 \leq i, j \leq p$, 
\begin{align*}
    \langle \phi_i, \phi_j \rangle
= \left\{
\begin{array}{ll}
1 & i=j;\\
 \\
\frac{1}{\sqrt{p}} \cdot \chi(a_i-a_j) & i\neq j.
\end{array}
\right.
\end{align*}
\end{lemma}

\begin{proof}
Since it is obvious when $i=j$, suppose that $i\neq j$.
For every $1 \leq k \leq (p-1)/2$, the equation $X^2 \equiv b_k \pmod{p}$ has exactly two distinct non-zero solutions by the definition of $b_k$.
Then it follows from Lemma~\ref{lem-Gauss} that
\begin{align*}
    \langle \phi_i, \phi_j \rangle
    &=\frac{1}{p}+\frac{2}{p}\sum_{k=1}^{\frac{p-1}{2}}\psi\bigl((a_i-a_j)b_k \bigr)\\
    &=\frac{1}{p}+\frac{2}{p}\cdot \frac{1}{2}\sum_{x \in \mathbb{F}_p^*}\psi\bigl((a_i-a_j)x^2 \bigr)\\
    &=\frac{1}{p}\sum_{x \in \mathbb{F}_p}\psi\bigl((a_i-a_j)x^2 \bigr)\\
    &=\frac{1}{p} \cdot \chi(a_i-a_j)\sqrt{p}=\frac{1}{\sqrt{p}} \cdot \chi(a_i-a_j).
\end{align*}
\end{proof}
The next definition is for the Paley graph extractor.
\begin{definition}[Paley graph extractor, \cite{CG1988}]
\label{def-PaleyExt}
For an odd prime $p$, the {\it Paley graph extractor} $\Ext_p : \F_p \times \F_p \to \{0,1\}$ is a function defined as 
    \begin{equation*}
\Ext_p(x,y):=
    \begin{cases}
    1  & x=y;\\
    \frac{1}{2}(\chi(x-y)+1)  & \text{$x\neq y$}.
    \end{cases}
\end{equation*}
\end{definition}

The following property is convenient for investigating the Paley graph extractor.

\begin{definition}
\label{def-pgc}
Let $p$ be an odd prime.
For $0<\alpha \leq 1$ and $\beta>0$, we say that the property $\mathcal{P}(\alpha, \beta)$ holds if for every pair of $S, T \subset \mathbb{F}_p$ with $|S|, |T|>p^{\alpha}$,
\begin{equation*}
\label{eq-pgc}
   \Bigl |\sum_{s \in S, t \in T}\chi(s-t) \Bigr|\leq C p^{-\beta}|S||T|
\end{equation*}
for some constant $C>0$.
\end{definition}

The following proposition proved in \cite{CG1988} plays a key role for investigating the Paley graph extractor.

\begin{proposition}[Lemma 5, Corollary 11 in \cite{CG1988}]
\label{prop-Ext}
Let $n=\lceil \log_2 p \rceil$.
If the property $\mathcal{P}(\alpha, \beta)$ holds for an odd prime $p$, then $\Ext_p$ is a $2$-source $(\alpha n, p^{-\beta})$-extractors.
\end{proposition}

\begin{remark}
\label{rem-pgc}
The well-known {\it Paley graph conjecture} tells that for each $0<\alpha \leq 1$ and any sufficiently large prime $p$, there exists a constant $\beta=\beta(\alpha)>0$ such that $\mathcal{P}(\alpha, \beta)$ holds; see e.g. \cite{C2008}, \cite{CG1988}, \cite{GM2020}.
Chor and Goldreich established Conjecture~\ref{conj-PaleyExt} based on the Paley graph conjecture.
\end{remark}

\begin{remark}
\label{rem-clique}
For a prime $p \equiv 1 \pmod{4}$, the {\it Paley graph} $G_p$ with $p$ vertices is defined as an undirected graph with vertex set $\mathbb{F}_p$ in which two distinct vertices $x$ and $y$ are adjacent if and only if $\chi(x-y)=1$.
The clique number (i.e. the size of the maximum cliques), denoted by $\omega(G_p)$, of $G_p$ has been extensively studied in graph theory and additive combinatorics, and it is conjectured that in fact $\omega(G_p)=O(\poly \log p)$.
The best known upper bound in \cite{HP2020} shows $\omega(G_p)\leq \sqrt{p/2}+1$. 
Notice that the property $\mathcal{P}(\alpha, \beta)$ for $\alpha<1/2$ implies a better bound that $\omega(G_p)\leq p^{\alpha}=o(\sqrt{p})$.
\end{remark}

\section{The first main result}
\label{sect-main}
Here is the first main result of this paper.
\begin{theorem}
\label{thm-main-intro}
Let $p\equiv 1 \pmod{4}$ be a prime number and $\varepsilon>0$ be an arbitrarily small fixed constant.
Suppose that the matrix $\Phi_p$ has the $(p^{1/2+\varepsilon}, p^{-\tau})$-RIP for some constant $\tau>0$.
Let $\gamma$ be any real number such that $0<\gamma<\tau$
Then there exists a constant $\beta=\beta(\varepsilon, \gamma)>0$ such that the property $\mathcal{P}(1/2-\tau+\gamma, \beta)$ holds.
In particular, according to the notion of Definition~\ref{def-pgc}, the real number $\beta$ can be taken as 
\begin{align}
\label{eq-beta}
    \beta=\left\{
\begin{array}{ll}
\gamma & p^{1/2-\tau+\gamma}<|S|, |T|\leq p^{1/2+\varepsilon};\\
0.05\varepsilon^2 & \text{either $|S|>p^{1/2+\varepsilon}$ or $|T|>p^{1/2+\varepsilon}$}.\\
\end{array}
\right.
\end{align}
\end{theorem}

\begin{remark}
\label{rem-1/2}
It holds that $0<1/2-\tau+\gamma<1/2$. Hence Theorem~\ref{thm-main-intro} significantly improved the state-of-the-art result (see Theorem~\ref{thm-Karatsuba} below) under the assumption for the RIP.
\end{remark}

\begin{remark}
Although it is still not known whether the assumption for the RIP in Theorem~\ref{thm-main-intro} is true, this is broadly believed to be true. 
Indeed, as pointed out in \cite{BFMW2013}, it matches both of theoretical and numerical results for pseudo-randomness of quadratic residues and estimating the clique number of Paley graphs;
see also Section~\ref{sect-RIPconst}.
\end{remark}

Keeping Remark~\ref{rem-1/2} in mind, we present several corollaries.
The following corollary is immediately obtained by Proposition~\ref{prop-Ext}.

\begin{corollary}
\label{cor-2}
Let $n=\lceil \log_2 p \rceil$.
Then, under the assumption of Theorem~\ref{thm-main-intro}, the Paley graph extractor $\Ext_p$ is 
a $2$-source extractors breaking the min-entropy $1/2$ barrier with exponentially small error.
Precisely, $\Ext_p$ is a $2$-source $(\alpha n, p^{-\beta})$-extractor for some $0<\alpha<1/2$ and $\beta>0$.
\end{corollary}

\begin{remark}
\label{rem-negligible}
Note that $p^{-\beta}$ is a negligible error since it is exponentially small as a function of $n=\lceil \log_2 p \rceil$.
\end{remark}

The next one follows from Remark~\ref{rem-clique}, which shows that Theorem~\ref{thm-main-intro} is an extension of the main theorem in \cite{BMM2017}.

\begin{corollary}
\label{cor-1}
Under the assumption of Theorem~\ref{thm-main-intro}, we have $\omega(G_p)=o(\sqrt{p})$.     
\end{corollary}

\section{Proof of Theorem~\ref{thm-main-intro}}
\label{sect-main-tech}

To prove Theorem~\ref{thm-main-intro}, we introduce the following two theorems. 
The first theorem is due to Karatsuba~\cite{K2008} and is well-known in number theory and computer science.

\begin{theorem}[\cite{K2008}]
\label{thm-Karatsuba}
Let $p$ be an odd prime and $\varepsilon>0$ be an arbitrarily small constant.
Then for any subsets $S, T \subset \F_p$ such that one has size more than $p^{\varepsilon}$ and another has size more than $p^{1/2+\varepsilon}$, we have 
\begin{equation*}
   \Bigl |\sum_{s \in S, t \in T}\chi(s-t) \Bigr|\leq p^{-0.05\varepsilon^2}|S||T|.
\end{equation*}
\end{theorem}

The next theorem is our contribution, and crucially relies on the RIP of $\Phi_p$.  

\begin{theorem}
\label{thm-main}
Let $p\equiv 1 \pmod{4}$ be a prime number and $\varepsilon>0$ be an arbitrarily small fixed constant.
Suppose that the matrix $\Phi_p$ has the $(p^{1/2+\varepsilon}, p^{-\tau})$-RIP for some constant $\tau>0$.
Let $\gamma$ be any real number such that $0<\gamma<\tau$.
Then for any subsets $S, T \subset \F_p$ with $p^{1/2-\tau+\gamma}<|S|, |T|\leq p^{1/2+\varepsilon}$, we have 
\begin{equation*}
   \Bigl |\sum_{s \in S, t \in T}\chi(s-t) \Bigr|
   \leq 4p^{-\gamma}|S||T|.
\end{equation*}
\end{theorem}

Assuming Theorems~\ref{thm-Karatsuba} and \ref{thm-main}, we present the proof of Theorem~\ref{thm-main-intro}.
The proof of Theorem~\ref{thm-main} is presented in the next section.

\begin{proof}[Proof of Theorem~\ref{thm-main-intro}]
The property $\mathcal{P}(1/2+\varepsilon, \beta)$ follows from Theorem~\ref{thm-Karatsuba} by taking $\beta=0.05\varepsilon^2$.
Since Theorem~\ref{thm-main} covers any subsets $S, T \subset \F_p$ with $p^{1/2-\tau+\gamma}<|S|, |T|\leq p^{1/2+\varepsilon}$, the property $\mathcal{P}(1/2-\tau+\gamma, \beta)$ follows, completing the proof.
\end{proof}

Now all that left is to do is to prove Theorem~\ref{thm-main}.
To this end, we show the following lemmas for computing the Rayleigh quotient for $\Phi_p$.

\begin{lemma}
\label{lem-directcomp}
For a subset $U \subset \F_p$, let $\mathbf{1}_U \in \{0,1\}^{p+1}$ be the indicator vector of $U$.  
Then for any subset $U \subset \F_p$ we have 
\begin{align*}
    \mathbf{1}_U^T(\Phi_p^T \overline{\Phi_p}) \mathbf{1}_U=|U|+\frac{1}{\sqrt{p}}\sum_{u, v\in U} \chi(u-v).
\end{align*}

\end{lemma}

\begin{proof}
Since $\Phi_p^T \overline{\Phi_p}$ is a Gram matrix, the lemma is straightforward by Lemma~\ref{lem-inner}.
Note that since $\chi(0)=0$, we have 
$$
\sum_{\substack{u, v\in U \\ u \neq v}} \chi(u-v)
=
\sum_{u, v\in U} \chi(u-v).
$$
\end{proof}

\begin{lemma}
\label{lem-RIPcomp}
Suppose that $\Phi_p$ has the $(K, \delta)$-RIP.
Then for any subset $U \subset \F_p$ with $|U|\leq K$,  we have 
\begin{align*}
    (1-\delta)|U| \leq \mathbf{1}_U^T(\Phi_p^T \overline{\Phi_p}) \mathbf{1}_U \leq (1+\delta)|U|.
\end{align*}
\end{lemma}

\begin{proof}
The lemma immediately follows from the assumption and a simple fact that $\mathbf{1}_U^T(\Phi_p^T \overline{\Phi_p})\mathbf{1}_U
=
(\Phi_p\mathbf{1}_U)^T (\overline{\Phi_p \mathbf{1}_U})
=
||\Phi_p\mathbf{1}_U||^2$.
\end{proof}

Lemmas~\ref{lem-directcomp} and \ref{lem-RIPcomp} straightforwardly yield the following lemma which plays a key role in the proof of Theorem~\ref{thm-main}.

\begin{lemma}
\label{lem-tech}
Suppose that $\Phi_p$ has the $(K, \delta)$-RIP.
Then for any subset $U \subset \F_p$ with $|U|\leq K$, we have
\begin{align*}
    \Biggl|\sum_{u, v\in U} \chi(u-v) \Biggr| \leq \delta p^{\frac{1}{2}} |U|
\end{align*}

\end{lemma}

Now we are ready to prove Theorem~\ref{thm-main}.

\begin{proof}[Proof of Theorem~\ref{thm-main}]
Suppose that the matrix $\Phi_p$ has the $(p^{1/2+\varepsilon}, p^{-\tau})$-RIP for some constant $\tau>0$.
Let $\gamma$ be any real number such that $0<\gamma<\tau$.
Let $S, T \subset \mathbb{F}_p$ be any subsets with $p^{1/2-\tau+\gamma}<|S|, |T|\leq p^{1/2+\varepsilon}$.
First observe that 
\begin{align*}
\label{eq-sum1}
    \sum_{s \in S, t \in T}\chi(s-t)=\sum_{u,v \in S \cup T}\chi(u-v)&-\sum_{u,v \in S \setminus T}\chi(u-v)\\&-\sum_{u,v \in T \setminus S}\chi(u-v)
\stepcounter{equation}\tag{\theequation}
\end{align*}
Here remark that the second or third term in the right-hand side of (\ref{eq-sum1}) does not contribute to the sum if $S \subset T$ or $T \subset S$, respectively.
By applying Lemma~\ref{lem-tech} for $U=S \cup T$, $U=S\setminus T$ and $U=T\setminus S$, the equation (\ref{eq-sum1}) implies that
\begin{align}
\label{eq-sum2}
    \Biggl| \sum_{s \in S, t \in T}\chi(s-t) \Biggr|
    \leq p^{\frac{1}{2}-\tau} \cdot (|S\cup T|+|S\setminus T|+|T \setminus S|).
\end{align}
Since $|S\cup T|\leq |S|+|T|$, $|S\setminus T|\leq |S|$ and $|T \setminus S|\leq |T|$, the inequality (\ref{eq-sum2}) yields 
\begin{align}
\label{eq-sum3}
    \Biggl| \sum_{s \in S, t \in T}\chi(s-t) \Biggr|
    \leq 2p^{\frac{1}{2}-\tau} \cdot (|S|+|T|).
\end{align}
By (\ref{eq-sum3}), it suffices to show that 
\begin{align*}
    2p^{\frac{1}{2}-\tau} \cdot (|S|+|T|) \leq 4p^{-\gamma}|S||T|,
\end{align*}
which immediately follows from the assumption that $|S|, |T|>p^{1/2-\tau+\gamma}$.
This completes the proof.    
\end{proof}

\section{The RIP constant of $\Phi_p$ and the second main result}
\label{sect-RIPconst}
In this section we briefly discuss a conjectured behavior of the RIP constant $\delta_K$ of the Paley ETF $\Phi_p$ as well as the behavior of the constant $\tau=\tau(K)$ in Theorem~\ref{thm-main-intro}.
In \cite{BFMW2013}, the following conjecture was proposed based on the model that the set of quadratic residues behaves like a random subsets of $\F_p$, supported by an estimation for the RIP constant of random partial (normalized) Fourier matrices, together with the fact that the Fourier portion of $\Phi_p$ is close to random partial (normalized) Fourier matrices with high probability (as $p\to \infty$); for details, see Section 5.3 in \cite{BFMW2013}.

\begin{conjecture}[Conjecture 25 in \cite{BFMW2013}]
\label{conj-RIPconst}
For a prime $p\equiv 1 \pmod{4}$ and $K\geq 1$ the RIP constant $\delta_K$ of $\Phi_p$ satisfies that
\begin{align*}
    \delta_K=O\Biggl(\sqrt{\frac{K}{p}} \cdot \log K \cdot \log p \Biggr).
\end{align*}
Accordingly, $\Phi_p$ has the $(\frac{c_1 p}{\log^{c_2} p}, \delta)$-RIP for constant numbers $c_1, c_2>0$ and $0<\delta<\sqrt{2}-1$.
\end{conjecture}

Although this conjecture is still open (and seems to be very hard to prove), it is broadly believed and agrees a widely-believed conjecture for the clique number $\omega(G_p)$ of the Paley graph $G_p$, stating that $\omega(G_p)={\rm poly} \log p$.

For $K=\Theta(p^{\theta})$, $\theta>0$, Conjecture~\ref{conj-RIPconst} immediately implies 
\begin{align*}
    \delta_K=O\Bigl(\theta \:p^{\frac{\theta-1}{2}} \log^2 p \Bigr).
\end{align*}
Hence, if Conjecture~\ref{conj-RIPconst} holds, then for any sufficiently large $p\equiv 1 \pmod{4}$ and $K=\Theta(p^{\theta})$, $\theta>0$, it is possible to take the real number $\tau=\tau(K)$ (in Theorem~\ref{thm-main-intro}) so that
\begin{align}
\label{eq-tau}
    \tau \leq \frac{1-\theta}{2}+o(1).
\end{align}


Assuming that Conjecture~\ref{conj-RIPconst} holds, then Theorem~\ref{thm-main-intro} and (\ref{eq-tau}) yield the property $\mathcal{P}(\alpha, \beta)$ for any $\alpha>0$ and some $\beta>0$.
Thus we have the following corollary which is closely related the Paley graph conjecture (see Remark~\ref{rem-pgc})

\begin{corollary}
\label{cor-3}
Suppose that Conjecture~\ref{conj-RIPconst} holds for sufficiently large primes $p\equiv 1 \pmod{4}$. 
Then the property $\mathcal{P}(\alpha, \beta)$ holds for {\rm any} $\alpha>0$ and some $\beta>0$.
\end{corollary}
\begin{proof}
Let $\alpha>0$ be given.
Since the case for $\alpha>1/2+\varepsilon$, $\varepsilon>0$ is an arbitrarily small fixed constant, was covered in Theorem~\ref{thm-Karatsuba}, we shall assume that $0<\alpha\leq 1/2+\varepsilon$.
By (\ref{eq-tau}) it is possible to take $\theta, \tau, \gamma$ with
\begin{align}
\label{eq-parameter1}
\left\{
\begin{array}{ll}
0<\theta<\alpha \\
 \\
\tau=\frac{1}{2}-\frac{3}{4}\theta>0\\
 \\
0<\gamma<\tau
\end{array}
\right.
\end{align}
so that
\begin{align}
\label{eq-parameter2}
\frac{1}{2}-\tau+\gamma=\frac{3}{4}\theta+\gamma<\theta<\alpha.
\end{align}   
Notice that Conjecture~\ref{conj-RIPconst} now implies that the Paley ETF $\Phi_p$ has the $(\Theta(p^\theta), p^{-\tau})$ for each $\theta>0$ and $\tau=1/2-3\theta/4$. 
Hence the same line of the proof of Theorem~\ref{thm-main} yields that for each $\theta>0$ and any subsets $S, T \subset \F_p$ with $p^{1/2-\tau+\gamma}<|S|, |T|\leq K=\Theta(p^{\theta})$ (which is valid thanks to (\ref{eq-parameter2})), 
\begin{equation*}
   \Bigl |\sum_{s \in S, t \in T}\chi(s-t) \Bigr|
   \leq 4p^{-\gamma}|S||T|.
\end{equation*}
Since $p^{1/2-\tau+\gamma}<p^\alpha$ by (\ref{eq-parameter2}), the same line of the proof of Theorem~\ref{thm-main-intro} yields the property $\mathcal{P}(\alpha, \beta)$ where $\beta>0$ can be chosen as (\ref{eq-beta}) (and $\gamma$ is chosen so that (\ref{eq-parameter1}) and (\ref{eq-parameter2})).
\end{proof}

Finally, Proposition~\ref{prop-Ext} yields the following corollary, which is the second main result of this paper.

\begin{corollary}
\label{cor-4}
Suppose that Conjecture~\ref{conj-RIPconst} holds for sufficiently large primes $p\equiv 1 \pmod{4}$.
Let $n=\lceil \log_2 p \rceil$.
Then the Paley graph extractor $\Ext_p$ is a $2$-source $(\alpha n, p^{-\beta})$-extractor for {\rm any} $0<\alpha<1/2$ and some $\beta>0$, which is the statement of Conjecture~\ref{conj-PaleyExt}.
\end{corollary}


As a byproduct of Corollary~\ref{cor-4}, Conjecture~\ref{conj-RIPconst} also implies an upper bound for the clique number $\omega(G_p)$, say, $\omega(G_p)=o(p^{\alpha})$ for any $\alpha>0$; recall Remark~\ref{rem-clique}.




\section*{Acknowledgements}
This work was supported by JST, ACT-X Grant Number JPMJAX2109, Japan. This paper was written during an academic visit to Department of Mathematics, Emory University.
We appreciate Ayush Basu, Yujie Gu, Vojt\v{e}ch R\"{o}dl, Tsuyoshi Takagi and Kenji Yasunaga for valuable discussions.

\end{document}